\documentclass[12pt,thmsa]{article}
\usepackage{amssymb}
\usepackage{amsfonts}
\usepackage{amsmath}
\usepackage[top=3.8cm,bottom=3.8cm,textwidth=16cm,centering,a4paper]{geometry}
\usepackage{verbatim}
\usepackage[numbers,sort&compress]{natbib}
\usepackage{color}
\usepackage{indentfirst}

\allowdisplaybreaks[4]
\newtheorem{theorem}{Theorem}[section]

\newtheorem{remark}[theorem]{Remark}

\newtheorem{lemma}[theorem]{Lemma}

\newtheorem{definition}[theorem]{Definition}
\newenvironment{proof}[1][Proof]{\noindent\textbf{#1.} }{\ \rule{0.5em}{0.5em}}

\begin{document}

\title{Normalized solutions for the NLS equation with potential in higher
dimension: the purely Sobolev critical case}
\date{}
\author{Juntao Sun$^{a}$\thanks{%
E-mail address: jtsun@sdut.edu.cn(J. Sun)}, Shuai Yao$^{a}$\thanks{%
E-mail address: shyao@sdut.edu.cn (S. Yao)}, He Zhang$^{b}$\thanks{
E-mail address: hehzhang@163.com(H. Zhang)} \\
{\footnotesize $^{a}$\emph{School of Mathematics and Statistics, Shandong
University of Technology, Zibo 255049, PR China }}\\
{\footnotesize $^b$\emph{School of Mathematics and Statistics, Central South
University, Changsha 410083, PR China }}}
\maketitle

\begin{abstract}
We study normalized solutions for the nonlinear Schr\"{o}dinger (NLS) equation with potential and Sobolev critical nonlinearity. By establishing suitable assumptions on the potential, together with new
techniques, we find a mountain-pass type solution for $N\geq 6$, which solves an open problem presented
in a recent paper [Verzini and Yu, arXiv:2505.05357v1]. Moreover, we also find a local minimizer with negative energy for $N\geq 3$, which improves the results in [Verzini and Yu, arXiv:2505.05357v1].
\end{abstract}



\section{Introduction}

Consider the problem
\begin{equation}
\left\{
\begin{array}{ll}
-\Delta u+V(x)u=\lambda u+|u|^{2^{\ast }-2}u & \ \text{in}\ \mathbb{R}^{N},
\\
\int_{\mathbb{R}^{N}}u^{2}dx=a^{2}, &
\end{array}%
\right.  \label{e1}
\end{equation}%
where $N\geq 3,a>0,\lambda \in \mathbb{R}$ is unknown, $2^{\ast }:=\frac{2N}{%
N-2}$ is the Sobolev critical exponent, and $V$ is a potential.

From a variational point of view, solutions of problem (\ref{e1}) can be viewed as
critical points of the energy functional $I:H^{1}(\mathbb{R}^{N})\rightarrow
\mathbb{R}$ defined by
\begin{equation*}
I(u):=\frac{1}{2}\int_{\mathbb{R}^{N}}|\nabla u|^{2}dx+\frac{1}{2}\int_{%
\mathbb{R}^{N}}V(x)u^{2}dx-\frac{1}{2^{\ast }}\int_{\mathbb{R}%
^{N}}|u|^{2^{\ast }}dx
\end{equation*}%
on the constraint
\begin{equation}
\mathcal{M}_{a}:=\left\{ u\in H^{1}(\mathbb{R}^{N}):\int_{\mathbb{R}%
^{N}}u^{2}dx=a^{2}\right\} .  \label{e2}
\end{equation}%
Such solutions are usually called normalized solutions, since $\lambda \in
\mathbb{R}$ appears as a Lagrange multiplier and $L^{2}$-norms of solutions
are prescribed. This study often offers a good insight of the dynamical
properties of solutions for the associated evolutive equation, such as
stability or instability.

After the pioneering contribution by Lions \cite{L84} and Jeanjean \cite{J97}%
, the study of normalized solutions for NLS equation (or system) with
different types of potentials and nonlinearities has attracted much
attention in recent years. We refer to \cite{BM21,JL22,PV17,PVY24,S120,S20,SZ24,WW22,ZZ23} for the non-potential case and to
\cite{BMRV21,DZ22,IM20,MRV22,PPVV21,SZRW25,SYZ24,WS23,VY25,ZYS25} for the potential case.

For the potential case, finding normalized solutions seems to be more
difficult, since the techniques developed in the study of the non-potential
case, can not applied directly. In view of this, besides the importance in
the applications, not negligible reasons of many researchers' interest for
such problems are its stimulating and challenging difficulties. Let us
briefly comment some well-known results.

Case $(i):$ The mass-subcritical nonlinearity. Ikoma and Miyamoto \cite{IM20}
used the standard concentration compactness arguments to find a global
minimizer with negative energy when the potential $V(x)$ is negative and
vanishes on infinity, i.e. $\lim\limits_{|x|\rightarrow +\infty }V(x)=0.$
Later, when $\lim\limits_{|x|\rightarrow +\infty }V(x)=\sup_{x\in \mathbb{R}%
^{N}}V(x)\in (0,+\infty ]$, Zhong and Zou \cite{ZZ23} established the strict
sub-additive inequality, which can greatly simplify the discussion process
in the traditional sense.

Case $(ii):$ The mass-supcritical and Sobolev subcritical nonlinearity. Ding
and Zhong \cite{DZ22} found a local minimizer with positive energy when the
potential $V(x)\in C^{1}(\mathbb{R}^{N})$ satisfies $\lim\limits_{|x|%
\rightarrow +\infty }V(x)=\sup_{x\in \mathbb{R}^{N}}V(x)=0$ and some other
assumptions. If the potential $V(x)$ is positive and includes singularities,
then the mountain pass structure by Jeanjean \cite{J97} is destroyed. In
such a case, Bartsch et al. \cite{BMRV21} used a new variational principle
exploiting the Pohozaev identity and obtained a solution with high Morse
index and positive energy by constructing a suitable linking structure. If
the potential $V(x)$ is negative and satisfies $-V(x)\leq
\limsup_{|x|\rightarrow +\infty }-V(x)<+\infty ,$ Molle et al. \cite{MRV22}
proved the existence of two solutions, where one is a local minimizer with
negative energy and the other one is a mountain-pass type solution with
positive energy.

Case $(iii):$ The mass-supcritical and Sobolev critical nonlinearity. In a
recent and interesting paper, Verzini and Yu \cite{VY25} considered the
purely Sobolev critical case, i.e. problem (\ref{e1}). They allowed the
potential $V(x)$ to be irregular and weakly attractive. More precisely, $%
V(x) $ satisfies the following assumptions:
\begin{eqnarray}
V(x) :&=&V_{1}(x)+V_{2}(x)\text{ with}\ V_{1}\in L^{N/2}(\mathbb{R}^{N}),
\notag \\
W(x) :&=&V(x)|x|=W_{1}(x)+W_{2}(x)\text{ with }W_{1}\in L^{N}(\mathbb{R}%
^{N}),  \notag \\
V_{2},W_{2} &\in &L^{\infty }(\mathbb{R}^{N})\text{ and }%
|V_{2}(x)|+|W_{2}(x)|\rightarrow 0\ \text{as}\ |x|\rightarrow \infty ,
\label{e3}
\end{eqnarray}%
and
\begin{equation}
\widetilde{\lambda }_{1}:=\inf \left\{ \int_{\mathbb{R}^{N}}\left( |\nabla
\psi |^{2}+V(x)\psi ^{2}\right) dx:\psi \in \mathcal{M}_{1}\right\} <0,
\label{e4}
\end{equation}%
where $\mathcal{M}_{1}=\mathcal{M}_{a}$ with $a=1$ in (\ref{e2}). Under (%
\ref{e3})--(\ref{e4}) and some additional assumptions on $V$, they
concluded that problem (\ref{e1}) admits a local minimizer with negative
energy for all $N\geq 3,$ and a mountain-pass type solution with positive
energy for $3\leq N\leq 5$. The results can be regarded as paralleling those
in \cite{S20,JL22,WW22}, in which NLS equation with combined nonlinearities $(u)=|u|^{2^{\ast }-2}u+b|u|^{q-2}u$ ($2<q<2^{\ast }$ and $b>0$) was studied. The presence of the quadratic potential term $\int_{\mathbb{R}^{N}}V(x)u^{2}dx$, instead of the superquadratic perturbation $-b\int_{\mathbb{R}^{N}}|u|^{q}dx$, reflects on several different difficulties.

In \cite[Remark 1.11]{VY25}, it was pointed out that the existence of a
mountain-pass solution for $N\geq 6$ remains an interesting open problem.
Motivated by this fact, in this paper our first aim is to solve this open
problem, and to establish suitable assumptions on $V,$ together with the new
approach, such that a mountain-pass solution exists for $N\geq 6$. In
addition, our second aim is to restudy the existence of local minimizer by
weaking the assumptions on $V$ in \cite[Theorem 1.3]{VY25}.

Before stating our main results, we need the following definition and some
notations.

\begin{definition}
\label{T1.1}We say that a solution $u_{0}\in \mathcal{M}_{a}$ of problem (%
\ref{e1}) is a ground state if it possesses the minimal energy among all
solutions in $\mathcal{M}_{a}$, i.e. if
\begin{equation*}
I(u_{0})=\inf \left\{ I(u):u\in \mathcal{M}_{a},\text{ }(I|_{\mathcal{M}%
_{a}})^{\prime }(u)=0\right\} .
\end{equation*}
\end{definition}

For $N\geq 3$, we denote by $L^{p}(\mathbb{R}^{N})$ the Lebesgue space with
the norm $\Vert \cdot \Vert _{p}$ and by $H^{1}(\mathbb{R}^{N})$ the usual
Sobolev space with the norm $\Vert \cdot \Vert _{H^{1}}^{2}=\Vert \nabla
\cdot \Vert _{2}^{2}+\Vert \cdot \Vert _{2}^{2}$, while $D^{1,2}(\mathbb{R}%
^{N})$ stands for the homogenous Sobolev space with norm $\Vert \nabla \cdot
\Vert _{2}$. Denoting with $\mathcal{S}$ be the best Sobolev constant of the
embedding $D^{1,2}(\mathbb{R}^{N})\hookrightarrow L^{2^{\ast }}(\mathbb{R}%
^{N})$, we have%
\begin{equation*}
\mathcal{S=}\inf_{u\in H^{1}(\mathbb{R}^{N})\backslash \{0\}}\frac{\Vert
\nabla u\Vert _{2}^{2}}{\Vert u\Vert _{2^{\ast }}^{2}}=\min_{u\in D^{1,2}(%
\mathbb{R}^{N})\backslash \{0\}}\frac{\Vert \nabla u\Vert _{2}^{2}}{\Vert
u\Vert _{2^{\ast }}^{2}}.
\end{equation*}

Our main results are presented as follows.

\begin{theorem}
\label{T1.2} Assume that $N\geq 3$ and conditions (\ref{e3})--(\ref{e4})
hold. If
\begin{equation}
\Vert V_{1}^{-}\Vert _{N/2}<\mathcal{S}\ \text{and}\ \Vert V_{2}^{-}\Vert
_{\infty }<\frac{2}{N}\left( 1-\Vert V_{1}^{-}\Vert _{N/2}\mathcal{S}%
^{-1}\right) ^{\frac{N}{2}}a^{-2}\mathcal{S}^{\frac{N}{2}},  \label{e5}
\end{equation}%
then problem (\ref{e1}) has a solution $\bar{v}\in H^{1}(\mathbb{R}^{N})$,
which correspondis to a local minimizer of $I$ on $\mathcal{M}_{a}$ with
negative energy, negative multipliers, and that can be chosen to be strictly
positive almost everywhere in $\mathbb{R}^{N}$.\newline
If in addition,
\begin{equation}
\max \left\{ \Vert W_{1}\Vert _{N},a^{2}\Vert W_{2}\Vert _{\infty }\right\}
<C_{1},\ \text{if}\ N=3,4  \label{e6}
\end{equation}%
or
\begin{equation}
\max \left\{ \Vert V_{1}^{+}\Vert _{N/2},\Vert W_{1}\Vert _{N}\right\} <C_{2}%
\text{ and }\max \left\{ a^{2}\Vert V_{2}^{+}\Vert _{\infty },a\Vert
W_{2}\Vert _{\infty }\right\} <C_{3},\ \text{if}\ N\geq 5,  \label{e7}
\end{equation}%
for appropriate constants $C_{1},C_{2},C_{3}>0$, then the local minimizer $%
\bar{v}$ is a ground state.
\end{theorem}

\begin{remark}
\label{T1.3} In Theorem \ref{T1.2}, we likewise find a local minimizer in a
small sphere of $\mathcal{M}_{a}$ similar to that in \cite[Theorem 1.3]{VY25}.
However, our main innovation lies in allowing the radius of the sphere to
be controlled by $V$ rather than being a fixed constant, which can help us
weak the assumptions on $V$ in \cite[(1.8)]{VY25}. Furthermore, when we
prove that the local minimizer is a ground state, there is no need to impose
separate constraints on $V^{+}$ for $N=3,4,$ i.e. (\ref{e6}), which can
weak the assumptions in \cite[(1.9) and (1.10)]{VY25}.
\end{remark}

Next, we study the existence of mountain-pass solution for problem (\ref%
{e1}) with $N\geq 6$. To find a mountain-pass solution, we need to
establish the following assumption on $V:$%
\begin{equation}
\inf\limits_{x\in B_{2R_{\ast }}}V_{1}^{-}>\Vert V_{2}^{-}\Vert _{\infty
}\quad \text{for some }R_{\ast }>0. \label{e8}
\end{equation}%
Condition (\ref{e8}) is crucial for the energy estimation of mountain-pass
level.

\begin{theorem}
\label{T1.4} Let $N\geq 6.$ Assume that all assumptions of Theorem \ref{T1.2}
and (\ref{e8}) are satisfied. Then problem (\ref{e1}) has a
mountain-pass solution $\bar{v}_{\ast }\in H^{1}(\mathbb{R}^{N})$ at
positive energy level, with negative Lagrange multiplier, which can be
chosen to be strictly positive almost everywhere in $\mathbb{R}^{N}$.
\end{theorem}

\begin{remark}
\label{T1.5}$(i)$ According to the assumptions on $V$ in Theorems \ref{T1.2}
and \ref{T1.4}, we find that $V^{-}\neq 0$ needs to be sufficiently large.
Moreover, the case of $V^{+}\neq 0$ may exist.\newline
$(ii)$ We easily find some examples of $V$ satisfying all assumptions in
Theorems \ref{T1.2} and \ref{T1.4}. For example, let
\begin{equation*}
V(x)=-\frac{C}{1+|x|^{\tau }},
\end{equation*}
where ${\tau }>2$ and $C>0$. By calculating, we deduce that
\begin{eqnarray}
\Vert V\Vert _{N/2} &=&C(Nw_{N})^{\frac{2}{N}}\left(\int_{0}^{+\infty }\frac{%
s^{N-1}}{(1+s^{\tau })^{\frac{N}{2}}}ds\right)^{\frac{2}{N}}<+\infty ,  \notag \\
\Vert W\Vert _{N} &=&\left( \int_{\mathbb{R}^{N}}|V(x)|^{\frac{N}{2}}|V(x)|^{%
\frac{N}{2}}|x|^{N}dx\right) ^{\frac{1}{N}}  \notag \\
&=&\left( \int_{\mathbb{R}^{N}}|V(x)|^{\frac{N}{2}}\frac{C^{\frac{N}{2}%
}|x|^{N}}{\left( 1+|x|^{\tau }\right) ^{\frac{N}{2}}}dx\right) ^{\frac{1}{N}}
\notag \\
&\leq &C^{\frac{1}{2}}\Vert V\Vert _{N/2}^{\frac{1}{2}}<+\infty ,
\label{e9}
\end{eqnarray}%
where $w_{N}$ is the volume of the $N$-dimensional unit sphere. Let
\begin{eqnarray}
&&V_{1}(x)=V(x)\cdot \chi _{B_{R_{1}}},\ V_{2}(x)=V(x)\cdot \chi _{\mathbb{R}%
^{N}\setminus B_{R_{1}}},  \notag \\
&&W_{1}(x)=W(x)\cdot \chi _{B_{R_{2}}},\ W_{2}(x)=W(x)\cdot \chi _{\mathbb{R}%
^{N}\setminus B_{R_{2}}},  \label{e10}
\end{eqnarray}%
for some $R_{1},R_{2}>0$. By (\ref{e9})--(\ref{e10}) and the definition
of $V$, we easily verify that conditions (\ref{e3})--(\ref{e4}) and (\ref%
{e8}) hold. Moreover, by controlling $C>0$ sufficiently small, $%
R_{1},R_{2}>0$ sufficiently large, or the mass $a>0$ sufficiently small,
conditions (\ref{e5})--(\ref{e7}) hold.
\end{remark}

This paper is organized as follows. In Section 2, we present some
preliminary results and prove Theorem \ref{T1.2}. Section 3 is devoted to
proving Theorem \ref{T1.4}.

\section{The existence of local minimizer}

In this section, we firstly give some preliminary results. For convenience
of calculation we apply the transformation%
\begin{equation*}
v=\frac{1}{a}u\text{ and }\mu =a^{\frac{4}{N-2}}>0,
\end{equation*}%
to convert problem (\ref{e1}) into the following one, which also
incorporates the sign condition:
\begin{equation}
\left\{
\begin{array}{ll}
-\Delta v+V(x)v=\lambda v+\mu |v|^{2^{\ast }-2}v & \ \text{in}\ \mathbb{R}%
^{N}, \\
v\geq 0,\ \int_{\mathbb{R}^{N}}v^{2}dx=1. &
\end{array}%
\right.  \label{e11}
\end{equation}%
Then the solutions of problem (\ref{e11}) correspond to critical points of
the energy functional
\begin{equation*}
J_{\mu }(v)=\frac{1}{2}\int_{\mathbb{R}^{N}}|\nabla v|^{2}dx+\frac{1}{2}%
\int_{\mathbb{R}^{N}}V(x)v^{2}dx-\frac{\mu }{2^{\ast }}\int_{\mathbb{R}%
^{N}}|v|^{2^{\ast }}dx
\end{equation*}%
constrained to the $L^{2}$-norm
\begin{equation*}
\mathcal{M}:=\mathcal{M}_{1}=\left\{ v\in H^{1}(\mathbb{R}^{N}):\int_{%
\mathbb{R}^{N}}v^{2}dx=1\right\} .
\end{equation*}

\begin{lemma}
\label{T2.1}(\cite[Lemma 2.2]{MRV22}) For every $v\in H^{1}(\mathbb{R}^{N})$%
, we have,
\begin{eqnarray*}
&&\left\vert \int_{\mathbb{R}^{N}}V_{1}(x)v^{2}dx\right\vert \leq \Vert
V_{1}\Vert _{N/2}\Vert v\Vert _{2^{\ast }}^{2}\leq \mathcal{S}^{-1}\Vert
V_{1}\Vert _{N/2}\Vert \nabla v\Vert _{2}^{2}, \\
&&\left\vert \int_{\mathbb{R}^{N}}V_{1}(x)v(x)\nabla v(x)\cdot
xdx\right\vert \leq \Vert W_{1}\Vert _{N}\Vert v\Vert _{2^{\ast }}\Vert
\nabla v\Vert _{2}\leq \mathcal{S}^{-1/2}\Vert W_{1}\Vert _{N}\Vert \nabla
v\Vert _{2}^{2}, \\
&&\left\vert \int_{\mathbb{R}^{N}}V_{2}(x)v^{2}dx\right\vert \leq \Vert
V_{2}\Vert _{\infty }\Vert v\Vert _{2}^{2}, \\
&&\left\vert \int_{\mathbb{R}^{N}}V_{2}(x)v(x)\nabla v(x)\cdot
xdx\right\vert \leq \Vert W_{2}\Vert _{\infty }\Vert v\Vert _{2}\Vert \nabla
v\Vert _{2}.
\end{eqnarray*}
\end{lemma}

\begin{lemma}
\label{T2.2}(\cite[Lemma 3.1]{VY25}) Let $v\in \mathcal{M}$ be fixed. Then
we have\newline
$(i)$ $\Vert \nabla v_{t}\Vert _{2}\rightarrow 0$ and $J_{\mu }\left(
v_{t}\right) \rightarrow 0$ as $t\rightarrow 0$, where $v_{t}=t^{N/2}v(tx)$
for $t>0;$\newline
$(ii)$ $\Vert \nabla v_{t}\Vert _{2}\rightarrow +\infty $ and $J_{\mu
}\left( v_{t}\right) \rightarrow -\infty $ as $t\rightarrow +\infty .$
\end{lemma}

\begin{lemma}
(\cite[Lemma 2.6]{VY25})\label{T2.3} Let $(v_{n},\lambda _{n})$ be a
Palais-Smale sequence for $J_{\mu }$ on $\mathcal{M}$:
\begin{equation*}
\{v_{n}\}\subset \mathcal{M},\ J_{\mu }\left( v_{n}\right) \rightarrow c\in
\mathbb{R},\ \Vert J_{\mu }^{\prime }\left( v_{n}\right) \Vert \rightarrow 0,
\end{equation*}%
as $n\rightarrow +\infty $, and assume that $\left( v_{n},\lambda
_{n}\right) $ is bounded and $\Vert v_{n}^{-}\Vert _{H^{1}}\rightarrow 0$.
Then, up to subsequences, there exist $v_{0}\in H^{1}$, $v_{0}\geq 0$, and $%
\lambda _{0}\in \mathbb{R}$ such that $v_{n}\rightharpoonup v_{0}$ weakly in
$H^{1}(\mathbb{R}^{N})$, $\lambda _{n}\rightarrow \lambda _{0}$ and
\begin{equation*}
-\Delta v_{0}+V(x)v_{0}=\mu \left( v_{0}\right) ^{2^{\ast }-1}+\lambda
_{0}v_{0}.
\end{equation*}%
Moreover, only one of the three following alternatives holds:\newline
$(i)$ either $\lambda _{0}<0$ and $v_{n}\rightarrow v_{0}$ strongly in $%
H^{1}(\mathbb{R}^{N})$,\newline
$(ii)$ or $\lambda _{0}<0$, $v_{n}\rightharpoonup v_{0}$ in $H^{1}(\mathbb{R}%
^{N})$ (but not strongly), and
\begin{equation*}
J_{\mu }\left( v_{n}\right) \geq J_{\mu }\left( v_{0}\right) +\frac{1}{N}%
\mathcal{S}^{\frac{N}{2}}\mu ^{1-\frac{N}{2}}+o(1),\ \text{as}\ n\rightarrow
+\infty ,
\end{equation*}%
$(iii)$ or $\lambda _{0}\geq 0$, and in such case $v_{0}\equiv 0$.
\end{lemma}

\begin{lemma}
\label{T2.4} Assume that conditions (\ref{e3})--(\ref{e5}) hold. Then
there exists
\begin{equation}
\overline{t}_{\mu }>\left( \frac{1-\Vert V_{1}^{-}\Vert _{N/2}\mathcal{S}%
^{-1}}{\mu \mathcal{S}^{-\frac{N}{N-2}}}\right) ^{\frac{N-2}{4}}
\label{e12}
\end{equation}%
such that
\begin{equation*}
\beta _{\mu }:=\inf\limits_{\Vert v\Vert _{2}^{2}\leq 1}\sup \left\{ J_{\mu
}(v):0\leq \Vert \nabla v\Vert _{2}\leq \overline{t}_{\mu }\right\} >0.
\end{equation*}
\end{lemma}

\begin{proof}
For each $v\in H^{1}(\mathbb{R}^{N})$ and $\Vert v\Vert _{2}^{2}\leq 1$, it
follows from Lemma \ref{T2.1} that
\begin{eqnarray}
J_{\mu }(v) &\geq &\frac{1}{2}\Vert \nabla v\Vert _{2}^{2}-\frac{\mu }{%
2^{\ast }}\mathcal{S}^{-\frac{N}{N-2}}\Vert \nabla v\Vert _{2}^{2^{\ast}}-\frac{1}{2}\mathcal{S}^{-1}\Vert V_{1}^{-}\Vert _{N/2}\Vert
\nabla v\Vert _{2}^{2}-\frac{1}{2}\Vert V_{2}^{-}\Vert _{\infty }\Vert
v\Vert _{2}^{2}  \notag \\
&\geq &\frac{1}{2}\left( 1-\Vert V_{1}^{-}\Vert _{N/2}\mathcal{S}%
^{-1}\right) \Vert \nabla v\Vert _{2}^{2}-\frac{\mu }{2^{\ast }}\mathcal{S}%
^{-\frac{N}{N-2}}\Vert \nabla v\Vert ^{2^{\ast}}_{2}-\frac{1}{2}\Vert
V_{2}^{-}\Vert _{\infty }  \notag \\
&=&f(\Vert \nabla v\Vert _{2}),  \label{e13}
\end{eqnarray}%
where%
\begin{equation*}
f(t):=\frac{1}{2}\left( 1-\Vert V_{1}^{-}\Vert _{N/2}\mathcal{S}^{-1}\right)
t^{2}-\frac{\mu }{2^{\ast }}\mathcal{S}^{-\frac{N}{N-2}}t^{2^{\ast }}-\frac{1%
}{2}\Vert V_{2}^{-}\Vert _{\infty }\text{ for }t>0.
\end{equation*}%
According to condition (\ref{e5}), we calculate that
\begin{equation*}
f(t_{\ast })=\max\limits_{t>0}f(t)=\frac{1}{N}\left( 1-\Vert V_{1}^{-}\Vert
_{N/2}\mathcal{S}^{-1}\right) ^{\frac{N}{2}}\mu ^{1-\frac{N}{2}}\mathcal{S}^{%
\frac{N}{2}}-\frac{1}{2}\Vert V_{2}^{-}\Vert _{\infty }>0,
\end{equation*}%
where%
\begin{equation}
t_{\ast }:=\left( \frac{1-\Vert V_{1}^{-}\Vert _{N/2}\mathcal{S}^{-1}}{\mu
\mathcal{S}^{-\frac{N}{N-2}}}\right) ^{\frac{N-2}{4}}.  \label{e14}
\end{equation}%
Moreover, there exist $\widetilde{t}_{\mu }<t_{\ast }<\overline{t}_{\mu }$
such that $f(\widetilde{t}_{\mu })=f(\overline{t}_{\mu })=0,$ $f(t)<0$ for
any $t\in \left( 0,\widetilde{t}_{\mu }\right) \cup \left( \overline{t}_{\mu
},\infty \right) $ and $f(t)>0$ for any $t\in \left( \widetilde{t}_{\mu },%
\overline{t}_{\mu }\right) $. Thus, it holds
\begin{equation}
f(t_{\ast })=\max\limits_{t\in (0,\overline{t}_{\mu })}f(t)>0.  \label{e15}
\end{equation}%
Using (\ref{e13}) and (\ref{e15}), we conclude that $\beta _{\mu }>0$. The
proof is complete.
\end{proof}

Set
\begin{equation*}
\alpha _{\mu }:=\inf \left\{ J_{\mu }(v):v\in \mathcal{M},\ \Vert \nabla
v\Vert _{2}\leq \overline{t}_{\mu }\right\} ,
\end{equation*}%
where $\overline{t}_{\mu }$ is as in (\ref{e12}). Then we have the
following lemma.

\begin{lemma}
\label{T2.5} Assume that conditions (\ref{e3})--(\ref{e5}) hold. Then $%
\alpha _{\mu }<0$ is achieved. Moreover, there holds $\alpha _{\mu }\geq -%
\frac{1}{2}\Vert V_{2}^{-}\Vert _{\infty }$.
\end{lemma}

\begin{proof}
Similar to the argument of \cite[Lemma 4.2]{VY25}, we can prove that $\alpha
_{\mu }<0$ is achieved. Since $\Vert \nabla v\Vert _{2}<\overline{t}_{\mu }$%
, by Lemma \ref{T2.4}, it follows that $\alpha _{\mu }\geq -\frac{1}{2}\Vert
V_{2}^{-}\Vert _{\infty }$. The proof is complete.
\end{proof}

Let $\zeta _{\mu }=\left\{ v\in \mathcal{M}:J_{\mu }(v)=\alpha _{\mu
}\right\} $. It follows from Lemmas \ref{T2.4} and \ref{T2.5} that $\zeta
_{\mu }$ contains the solution given in \cite[Theorem 1.3]{VY25}. In fact,
by Lemma \ref{T2.4}, we can see that if $f\left( \mathcal{S}^{\frac{N}{4}%
}\mu ^{-\frac{N-2}{4}}\right) >0$, then $\overline{t}_{\mu }>\mathcal{S}^{%
\frac{N}{4}}\mu ^{-\frac{N-2}{4}}$. Combining with Lemmas \ref{T2.4} and \ref%
{T2.5}, we also derive that
\begin{equation}
\alpha _{\mu }=\inf \left\{ J_{\mu }(v):v\in \mathcal{M},\ \Vert \nabla
v\Vert _{2}\leq t_{\ast }\right\}
 ,  \label{e16}
\end{equation}%
where $t_{\ast }$ is as in (\ref{e14}). Then we show that the local
minimizer is a ground state.

\begin{lemma}
\label{T2.6} Assume that conditions (\ref{e3})--(\ref{e7}) hold. If $%
v\in H^{1}(\mathbb{R}^{N})$ satisfies
\begin{equation}
\left\{
\begin{array}{ll}
-\Delta v+V(x)v=\lambda v+\mu |v|^{2^{\ast }-2}v & \ \text{in}\ \mathbb{R}%
^{N}, \\
\Vert v\Vert _{2}=a\leq 1, &
\end{array}%
\right.  \label{e17}
\end{equation}%
with $J_{\mu }(v)<0$. Then $\Vert \nabla v\Vert _{2}<\overline{t}_{\mu }$
and $\lambda <0$. Furthermore, there holds%
\begin{equation*}
J_{\mu }(v)\geq \alpha _{\mu }\geq -\frac{1}{2}\Vert V_{2}^{-}\Vert _{\infty
}.
\end{equation*}
\end{lemma}

\begin{proof}
Since $v\in H^{1}(\mathbb{R}^{N})$ solves problem (\ref{e17}), it is clear
that $v$ satisfies the following Pohozaev identity
\begin{equation*}
\Vert \nabla v\Vert _{2}^{2}-\mu \Vert v\Vert _{2^{\ast }}^{2^{\ast }}+\frac{%
1}{2}\int_{\mathbb{R}^{N}}V(x)\left( Nv^{2}+2v\nabla v\cdot x\right) dx=0.
\end{equation*}%
Using this, together with $J_{\mu }(v)<0$, we deduce that
\begin{equation}
\frac{N-2}{2}\Vert \nabla v\Vert _{2}^{2}<\frac{N-4}{2}\Vert v\Vert
_{2^{\ast }}^{2^{\ast }}+\int_{\mathbb{R}^{N}}V(x)v\nabla v\cdot
xdx  \label{e18}
\end{equation}%
and
\begin{equation}
\frac{1}{N}\Vert \nabla v\Vert _{2}^{2}<\frac{N-4}{4}\int_{\mathbb{R}%
^{N}}V(x)v^{2}dx+\frac{N-2}{2N}\int_{\mathbb{R}^{N}}V(x)v\nabla v\cdot xdx.
\label{e19}
\end{equation}%
If $N=3,4$, by Lemma \ref{T2.1} and (\ref{e18}), we have
\begin{eqnarray*}
\frac{N-2}{2}\Vert \nabla v\Vert _{2}^{2} &<&\int_{\mathbb{R}%
^{N}}V(x)v\nabla v\cdot xdx \\
&\leq &\left( \mathcal{S}^{-1/2}\Vert W_{1}\Vert _{N}\Vert \nabla
v\Vert _{2}^{2}+\Vert W_{2}\Vert _{\infty }\Vert \nabla v\Vert _{2}\right) ,
\end{eqnarray*}%
which implies that
\begin{equation*}
\left( \frac{N-2}{2}-\mathcal{S}^{-1/2}\Vert W_{1}\Vert _{N}\right) \Vert
\nabla v\Vert _{2}<\Vert W_{2}\Vert _{\infty }.
\end{equation*}%
By (\ref{e6}), we choose an appropriate constant $C_{1}>0$ such that%
\begin{equation*}
\Vert W_{1}\Vert _{N}<\frac{N-2}{2}\mathcal{S}^{1/2}\text{ and }\frac{2\Vert
W_{2}\Vert _{\infty }}{N-2-2\mathcal{S}^{-1/2}\Vert W_{1}\Vert _{N}}<%
\overline{t}_{\mu }.
\end{equation*}%
If $N\geq 5$, by Lemma \ref{T2.1} and (\ref{e19}), one has
\begin{eqnarray}
&&\left( \frac{1}{N}-\frac{N-4}{4}\Vert V_{1}^{+}\Vert _{N/2}\mathcal{S}%
^{-1}-\frac{N-2}{2N}\mathcal{S}^{-1/2}\Vert W_{1}\Vert _{N}\right) \Vert
\nabla v\Vert _{2}^{2}  \notag \\
&<&\frac{N-2}{2N}\Vert W_{2}\Vert _{\infty }\Vert \nabla v\Vert _{2}+\frac{%
N-4}{4}\Vert V_{2}^{+}\Vert _{\infty }.  \label{e20}
\end{eqnarray}%
Define
\begin{equation}
h(t)=At^{2}-Bt-C\text{ for }t>0,  \label{e21}
\end{equation}%
where%
\begin{equation*}
\begin{array}{l}
A:=\frac{1}{N}-\frac{N-4}{4}\Vert V_{1}^{+}\Vert _{N/2}\mathcal{S}^{-1}-%
\frac{N-2}{2N}\mathcal{S}^{-\frac{1}{2}}\Vert W_{1}\Vert _{N}, \\
B:=\frac{N-2}{2N}\Vert W_{2}\Vert _{\infty }, \\
C:=\frac{N-4}{4}\Vert V_{2}^{+}\Vert _{\infty }.%
\end{array}%
\end{equation*}%
Then by (\ref{e7}), we choose appropriate constants $C_{2},C_{3}>0$ such
that $A>0$ and
\begin{equation*}
\widetilde{t}_{\ast }:=\frac{B}{2A}+\sqrt{\frac{B^{2}}{4A^{2}}+\frac{C}{A}}<%
\overline{t}_{\mu },
\end{equation*}%
where $\widetilde{t}_{\ast }$ satisfies $h\left( \widetilde{t}_{\ast
}\right) =0$ and $h\left( t\right) <0$ for any $t\in \left( 0,\widetilde{t}%
_{\ast }\right) $, and $\overline{t}_{\mu }$ is as in (\ref{e12}). So it
follows from (\ref{e20}) and (\ref{e21}) that $\Vert \nabla v\Vert _{2}<%
\overline{t}_{\mu }$.

Let $w=\frac{v}{a}$, then $\Vert w\Vert _{2}=1$. If $\Vert \nabla w\Vert
_{2}<\overline{t}_{\mu }$, then it is easy to obtain that $J_{\mu }(v)\geq
\alpha _{\mu }$. Otherwise, there exists $T\in \left( 1,\frac{1}{a}\right) $
such that $\Vert \nabla (Tv)\Vert _{2}=\overline{t}_{\mu }$. Then by Lemma %
\ref{T2.4}, we deduce that $J_{\mu }(Tu)>0,$ which contradicts with $J_{\mu
}(Tv)<T^{2}J_{\mu }(v)<0$. Hence, we have $J_{\mu }(v)\geq \alpha _{\mu }$.
Moreover, by \cite[Lemma 2.4]{VY25}, there holds $\lambda <0$. The proof is
complete.
\end{proof}

\begin{remark}
\label{T2.7} From Lemma \ref{T2.4}, we know that the size of $\overline{t}%
_{\mu }$ is decided by $\Vert V_{1}^{-}\Vert _{N/2}$ and $\Vert
V_{2}^{-}\Vert _{\infty }$. Moreover, we point out that $C_{1},C_{2}$ and $%
C_{3}$ given in Lemma \ref{T2.6} do exist. For instance, when $f\left(
\mathcal{S}^{\frac{N}{4}}\mu ^{-\frac{N-2}{4}}\right) \geq 0$, we can choose
appropriate constants $C_{1},C_{2},C_{3}>0$ such that $\Vert \nabla u\Vert
_{2}<\mathcal{S}^{\frac{N}{4}}\mu ^{-\frac{N-2}{4}}$. When $f\left( \mathcal{%
S}^{\frac{N}{4}}\mu ^{-\frac{N-2}{4}}\right) <0$, by (\ref{e16}), we can
choose appropriate constants $C_{1},C_{2},C_{3}>0$ such that $\Vert \nabla
u\Vert _{2}<t_{\ast }$. So the result of Lemma \ref{T2.6} can be extended to
that of \cite[Lemma 4.3]{VY25}.
\end{remark}

\textbf{We are ready to prove Theorem \ref{T1.2}:} Let $\{v_{n}\}$ be a
minimizing sequence for $\alpha _{\mu }$. Similar to the arguments of \cite[%
Lemmas 4.1 and 4.2]{VY25}, by using Lemmas \ref{T2.3} and \ref{T2.5}, we
conclude that there exists $\bar{v}\in \mathcal{M}$ such that up to a
subsequence, $v_{n}\rightarrow \bar{v}$ in $H^{1}(\mathbb{R}^{N})$ and $%
J_{\mu }\left( \bar{v}\right) =\alpha _{\mu }<0$. Furthermore, it follows
from Lemma \ref{T2.6} that $\bar{v}$ is a ground state.

\section{The existence of mountain-pass solution}

In this paper, we first introduce the minimax class
\begin{equation*}
\Gamma :=\left\{ \gamma \in C\left( [0,1],\mathcal{M}\right) ,\ \gamma (0)=%
\bar{v},\ \gamma (1)>\overline{t}_{\mu }\ \text{and}\ J_{\mu }\left( \gamma
(1)\right) <\alpha _{\mu }\right\} ,
\end{equation*}%
where $\bar{v}$ is the ground state solution of problem (\ref{e11})
satisfying $J_{\mu }(\bar{v})=\alpha _{\mu }$ and $\overline{t}_{\mu }$ is
as in (\ref{e12}). The family $\Gamma $ is not empty. Indeed, let $\bar{v}%
_{t}=t^{N/2}\bar{v}(tx)$ for $t>1$. By Lemmas \ref{T2.2} and \ref{T2.4},
there exists $t_{0},t_{1}>1$ such that
\begin{equation*}
\Vert \nabla \bar{v}_{t_{0}}\Vert _{2}<\overline{t}_{\mu }<\Vert \nabla \bar{%
v}_{t_{1}}\Vert _{2}\text{ and }J_{\mu }\left( \bar{v}_{t_{1}}\right)
<\alpha _{\mu }<0<J_{\mu }\left( \bar{v}_{t_{0}}\right) .
\end{equation*}%
So $\gamma _{\bar{v}}:T\in \lbrack 0,1]\mapsto \bar{v}_{1+T(t_{1}-1)}$ is a
path in $\Gamma $. Define the minimax value
\begin{equation*}
\theta _{\mu }=\inf\limits_{\gamma \in \Gamma }\sup\limits_{T\in \lbrack
0,1]}J_{\mu }\left( \gamma (T)\right) .
\end{equation*}%
Clearly, it holds
\begin{equation*}
\theta _{\mu }\geq \beta _{\mu }>0,
\end{equation*}%
where $\beta _{\mu }$ is given in Lemma \ref{T2.4}.

Now we estimate the minimax value $\theta _{\mu }$. Let us introduce the
function $U_{\varepsilon }\in H^{1}(\mathbb{R}^{N})$ defined as
\begin{equation*}
U_{\varepsilon }=\frac{\eta \left[ N(N-2)\varepsilon ^{2}\right] ^{\frac{N-2%
}{4}}}{\left[ \varepsilon ^{2}+|x|^{2}\right] ^{\frac{N-2}{2}}},
\end{equation*}%
where $\eta \in C_{0}^{\infty }(\mathbb{R}^{N})$, $0\leq \eta \leq 1$ is a
smooth cut-off function such that $\eta \equiv 1$ in $B_{R_{\ast }}$, and
$\eta \equiv 0$ in $\mathbb{R}^{N}\setminus B_{3R_{\ast }/2}$. Here, $R_{\ast }$ is as in (\ref{e8}). Then, by
using the estimates in \cite[eqs. (1.13), (1.29)]{BN83}, we obtain that
\begin{equation}
\begin{array}{l}
\Vert \nabla U_{\varepsilon }\Vert _{2}^{2}=\mathcal{S}^{N/2}+O(\varepsilon
^{N-2}), \\
\Vert U_{\varepsilon }\Vert _{2^{\ast }}^{2^{\ast }}=\mathcal{S}%
^{N/2}+O(\varepsilon ^{N}), \\
\Vert U_{\varepsilon }\Vert _{2}^{2}=c_{1}\varepsilon ^{2}+O(\varepsilon
^{N-2}),\ \text{if}\ N\geq 6.%
\end{array}
\label{e22}
\end{equation}%
Moreover, there holds
\begin{equation}
\Vert U_{\varepsilon }\Vert _{q}^{q}=\left\{
\begin{array}{ll}
c_{2}\varepsilon ^{N-\frac{N-2}{2}q}+o\left( \varepsilon ^{N-\frac{N-2}{2}%
q}\right)  & \ \text{for }\frac{N}{N-2}<q<2^{\ast }, \\
c_{3}\varepsilon ^{\frac{N}{2}}|\ln \varepsilon |+o\left( \varepsilon ^{%
\frac{N}{2}}\right)  & \ \text{for }q=\frac{N}{N-2}, \\
c_{4}\varepsilon ^{\frac{N-2}{2}q}+o\left( \varepsilon ^{\frac{N-2}{2}%
q}\right)  & \ \text{for }1\leq q<\frac{N}{N-2}.%
\end{array}%
\right.   \label{e23}
\end{equation}%
Next, we present some known results, which are crucial to estimate the value
of $\theta _{\mu }$.

\begin{lemma}
\label{T3.1} (\cite[Lemma 5.5]{SZ24}) For any $0<\phi \in L_{loc}^{\infty
}(B_{2R_{\ast }})$, we have as $\varepsilon \rightarrow 0^{+}$,
\begin{equation}
\begin{array}{l}
\int_{B_{2R_{\ast }}}\phi U_{\varepsilon }dx\leq 2[N(N-2)]^{\frac{N-2}{2}%
}\sup\limits_{x\in B_{2R_{\ast }}}\phi w_{N}R_{\ast }^{2}\varepsilon ^{\frac{%
N-2}{2}}+o\left( \varepsilon ^{\frac{N-2}{2}}\right) , \\
\int_{B_{2R_{\ast }}}\phi U_{\varepsilon }^{2^{\ast }-1}dx\geq \frac{1}{4}%
[N(N-2)]^{\frac{N+2}{2}}\inf\limits_{x\in B_{2R_{\ast }}}\phi
w_{N}\varepsilon ^{\frac{N-2}{2}}+o\left( \varepsilon ^{\frac{N-2}{2}%
}\right) .%
\end{array}
\label{e24}
\end{equation}
\end{lemma}

\begin{lemma}
\label{T3.2} (\cite[Lemma 5.3]{VY25}) Assume that $V\in L^{p}(\mathbb{R}%
^{N}) $ with $1\leq p<+\infty $, $q\in \mathbb{R}$ and $s>0$. Then there
holds%
\begin{equation*}
\left( \int_{\mathbb{R}^{N}}(s^{q}V\left( x/s\right) -V(x))^{p}dx\right)
^{1/p}\rightarrow 0\ \text{as}\ s\rightarrow 1.
\end{equation*}
\end{lemma}

We use a new method to estimate the mountain pass level $\theta _{\mu }$.

\begin{lemma}
\label{T3.3} Under the assumptions of Theorem \ref{T1.4}, there holds
\begin{equation*}
\theta _{\mu }<\alpha _{\mu }+\frac{1}{N}\mathcal{S}^{\frac{N}{2}}\mu ^{1-%
\frac{N}{2}}<\frac{1}{N}\mathcal{S}^{\frac{N}{2}}\mu ^{1-\frac{N}{2}},
\end{equation*}%
where $\alpha _{\mu }$ is the ground state energy level of problem (\ref%
{e11}).
\end{lemma}

\begin{proof}
The proof is divided into two steps.

\textbf{Step 1: We find a path $\gamma $ such that $\gamma \in \Gamma $. }%
From (\ref{e22}), it follows that $\Vert U_{\varepsilon }\Vert
_{2}^{2}\rightarrow 0$ as $\varepsilon \rightarrow 0$. Then for sufficiently
small $\varepsilon >0$, there holds $1-t^{2}\Vert U_{\varepsilon }\Vert
_{2}^{2}\geq 0$ for$\ t>0.$ Moreover, for sufficiently small $\varepsilon >0$%
, we define
\begin{equation*}
\psi _{\varepsilon ,t}=\left( 1-t^{2}\Vert U_{\varepsilon }\Vert
_{2}^{2}\right) ^{1/2}\bar{v}+tU_{\varepsilon }\ \text{for}\ t>0,
\end{equation*}%
where $\bar{v}$ is the ground state solution of problem (\ref{e11}) with $%
\lambda =\lambda _{\mu }$ such that $J_{\mu }(\bar{v})=\alpha _{\mu }$. It
is obvious that $\psi _{\varepsilon ,t}>0$ a.e. in $\mathbb{R}^{N}$. Set%
\begin{equation*}
\Psi _{\varepsilon ,t}=\xi ^{\frac{N-2}{2}}\psi _{\varepsilon ,t}(\xi x),
\end{equation*}%
where $\xi :=\Vert \psi _{\varepsilon ,t}\Vert _{2}>0$. A direct calculation
shows that $\Psi _{\varepsilon ,t}\in \mathcal{M}$ and
\begin{equation*}
\Vert \nabla \Psi _{\varepsilon ,t}\Vert _{2}^{2}=\Vert \nabla \psi
_{\varepsilon ,t}\Vert _{2}^{2},\ \Vert \Psi _{\varepsilon ,t}\Vert
_{2^{\ast }}^{2^{\ast }}=\Vert \psi _{\varepsilon ,t}\Vert _{2^{\ast
}}^{2^{\ast }},\ \Vert \Psi _{\varepsilon ,t}\Vert _{2}^{2}=\xi ^{-2}\Vert
\psi _{\varepsilon ,t}\Vert _{2}^{2}.
\end{equation*}%
Let $(\Psi _{\varepsilon ,t})_{s}=s^{N/2}\Psi _{\varepsilon ,t}(sx)$ for $%
s\geq 1$ and define the function
\begin{eqnarray*}
H_{t}(s) &=&J_{\mu }\left( \left( \Psi _{\varepsilon ,t}\right) _{s}\right)
\\
&=&\frac{s^{2}}{2}\Vert \nabla \Psi _{\varepsilon ,t}\Vert _{2}^{2}+\frac{%
s^{N}}{2}\int_{\mathbb{R}^{N}}V(x)|\Psi _{\varepsilon ,t}(sx)|^{2}dx-\frac{%
\mu s^{2^{\ast }}}{2^{\ast }}\Vert \Psi _{\varepsilon ,t}\Vert _{2^{\ast
}}^{2^{\ast }}\text{ for }s\geq 1.
\end{eqnarray*}%
Using Lemma \ref{T2.1}, (\ref{e22})--(\ref{e24}) and Minkowski inequality, one has
\begin{eqnarray*}
H_{t}(s) &=&\frac{s^{2}}{2}\Vert \nabla \bar{v}\Vert _{2}^{2}+\frac{s^{N}}{2}%
\int_{\mathbb{R}^{N}}V(x)|\Psi _{\varepsilon ,t}(sx)|^{2}dx-\frac{\mu s^{2^{\ast }}}{%
2^{\ast }}\Vert \bar{v}\Vert _{2^{\ast }}^{2^{\ast }}+\frac{s^{2}t^{2}}{2}\mathcal{S}^{\frac{N}{2}}-\frac{\mu s^{2^{\ast
}}t^{2^{\ast }}}{2^{\ast }}\mathcal{S}^{\frac{N}{2}}+o_{\varepsilon }(1)\\
&\leq&\frac{s^{2}}{2}\Vert \nabla \bar{v}\Vert _{2}^{2}-\frac{\mu s^{2^{\ast }}}{%
2^{\ast }}\Vert \bar{v}\Vert _{2^{\ast }}^{2^{\ast }}+\frac{s^{2}t^{2}}{2}\mathcal{S}^{\frac{N}{2}}-\frac{\mu s^{2^{\ast
}}t^{2^{\ast }}}{2^{\ast }}\mathcal{S}^{\frac{N}{2}}+s^2\Vert V_{1}\Vert _{N/2}\left(\Vert \bar{v}\Vert _{2^{\ast }}^{2}+
t^{2}\mathcal{S}^{\frac{N-2}{2}}\right)
\\
&&+\frac{1}{2}\Vert V_{2}\Vert _{\infty }+o_{\varepsilon }(1)
\end{eqnarray*}%
and
\begin{eqnarray*}
H_{t}^{\prime }(s) &=&s\Vert \nabla \bar{v}\Vert _{2}^{2}-\mu s^{2^{\ast
}-1}\Vert \bar{v}\Vert _{2^{\ast }}^{2^{\ast }}+\frac{N}{2}s^{N-1}\int_{%
\mathbb{R}^{N}}V(x)|\Psi _{\varepsilon ,t}(sx)|^{2}dx \\
&&+\frac{s^{N}}{2}\int_{\mathbb{R}^{N}}V(x)\Psi _{\varepsilon ,t}(sx)\nabla \Psi _{\varepsilon ,t}(sx)\cdot xdx+st^{2}\mathcal{S}^{\frac{N}{2}}-\mu s^{2^{\ast }-1}t^{2^{\ast
}}\mathcal{S}^{\frac{N}{2}}+o_{\varepsilon }(1) \\
&\leq &s\Vert \nabla \bar{v}\Vert _{2}^{2}-\mu s^{2^{\ast }-1}\Vert \bar{v}%
\Vert _{2^{\ast }}^{2^{\ast }}+st^{2}\mathcal{S}^{\frac{N}{2}}-\mu
s^{2^{\ast }-1}t^{2^{\ast }}\mathcal{S}^{\frac{N}{2}} +sN\Vert V_{1}\Vert _{N/2}\left(\Vert \bar{v}\Vert _{2^{\ast }}^{2}+
t^{2}\mathcal{S}^{\frac{N-2}{2}}\right)\\
&&+
\frac{N}{2s}\Vert V_{2}\Vert _{\infty }+\left[\frac{s}{2}\Vert W_{1}\Vert _{N}\left(\Vert \bar{v}\Vert _{2^{\ast }}^{2}+
t^{2}\mathcal{S}^{\frac{N-2}{2}}\right)+\frac{1}{2}\Vert W_{2}\Vert _{\infty }\right]\left(\Vert
\nabla \bar{v}\Vert _{2}+t\mathcal{S}^{\frac{N}{4}}\right)+o_{\varepsilon }(1).
\end{eqnarray*}%
From the above inequalities, we can choose $\overline{t}>0$ large enough and $\varepsilon >0$ small enough such
that $1-\overline{t}\Vert U_{\varepsilon }\Vert _{2}^{2}>0$ and $H_{%
\overline{t}}^{\prime }(s),H_{\overline{t}}(s)<0$ for $s\geq 1$. This shows
that
\begin{equation}
J_{\mu }\left( \left( \Psi _{\varepsilon ,\overline{t}}\right) _{s}\right)
\leq J_{\mu }\left( \Psi _{\varepsilon ,\overline{t}}\right) <0\text{ for}\
s\geq 1.  \label{e25}
\end{equation}%
Moreover, by Lemma \ref{T2.2} there exists $s_{0}>0$ such that
\begin{equation}
J_{\mu }\left( \left( \Psi _{\varepsilon ,\overline{t}}\right)
_{s_{0}+1}\right) <\alpha _{\mu }<0.  \label{e26}
\end{equation}%
Define%
\begin{equation*}
\gamma (T)=\left\{
\begin{array}{ll}
\Psi _{\varepsilon ,2T\overline{t}} & \text{ for }T\in \lbrack 0,1/2], \\
\left( \Psi _{\varepsilon ,\overline{t}}\right) _{2(T-1/2)s_{0}+1} & \text{
for }T\in \lbrack 1/2,1].%
\end{array}%
\right.
\end{equation*}%
Thus, it follows
from (\ref{e25})--(\ref{e26}) that $\gamma \in \Gamma $. This indicates
\begin{equation}
\theta _{\mu }\leq \sup_{t\in \lbrack 0,\overline{t}]}J_{\mu }(\Psi
_{\varepsilon ,t}).  \label{e27}
\end{equation}

\textbf{Step 2: We claim $\sup\limits_{t\in \lbrack 0,\overline{t}]}J_{\mu
}(\Psi _{\varepsilon ,t})<\alpha _{\mu }+\frac{1}{N}\mathcal{S}^{\frac{N}{2}%
}\mu ^{1-\frac{N}{2}}.$ } For convenience, let $\sigma =\left( 1-t^{2}\Vert
U_{\varepsilon }\Vert _{2}^{2}\right) ^{1/2}$. A direct calculation gives
that
\begin{eqnarray}
J_{\mu }(\Psi _{\varepsilon ,t}) &=&\frac{1}{2}\Vert \nabla \psi
_{\varepsilon ,t}\Vert _{2}^{2}+\frac{\xi ^{-2}}{2}\int_{\mathbb{R}%
^{N}}V(x/\xi )|\psi _{\varepsilon ,t}|^{2}dx-\frac{\mu }{2^{\ast }}\Vert
\psi _{\varepsilon ,t}\Vert _{2^{\ast }}^{2^{\ast }}  \notag \\
&=&J_{\mu }(\psi _{\varepsilon ,t})+\frac{\xi ^{-2}}{2}\int_{\mathbb{R}%
^{N}}\left( V(x/\xi )-\xi ^{2}V(x)|\psi _{\varepsilon ,t}|^{2}\right) dx
\notag \\
&=&J_{\mu }(\sigma \bar{v})+\frac{t^{2}}{2}\int_{\mathbb{R}^{N}}|\nabla
U_{\varepsilon }|^{2}dx+\frac{t^{2}}{2}\int_{\mathbb{R}^{N}}V(x)|U_{%
\varepsilon }|^{2}dx+t\sigma \int_{\mathbb{R}^{N}}V(x)\bar{v}U_{\varepsilon
}dx  \notag \\
&&+t\sigma \int_{\mathbb{R}^{N}}\nabla \bar{v}\nabla U_{\varepsilon }dx+%
\frac{\mu }{2^{\ast }}\int_{\mathbb{R}^{N}}\left( |\sigma \bar{v}|^{2^{\ast
}}-|\sigma \bar{v}+tU_{\varepsilon }|^{2^{\ast }}\right) dx  \notag \\
&&+\frac{\xi ^{-2}}{2}\int_{\mathbb{R}^{N}}\left( V(x/\xi )-\xi
^{2}V(x)\right) |\psi _{\varepsilon ,t}|^{2}dx.  \label{e28}
\end{eqnarray}%
Note that $(1+b)^{a}\geq 1+ab^{a-1}+b^{a}$ for any $a\geq 2,b>0$. Then there
holds
\begin{equation}
\frac{1}{2^{\ast }}\int_{\mathbb{R}^{N}}\left( |\sigma \bar{v}|^{2^{\ast
}}-|\sigma \bar{v}+tU_{\varepsilon }|^{2^{\ast }}\right) dx\leq -\frac{%
t^{2^{\ast }}}{2^{\ast }}\int_{\mathbb{R}^{N}}|U_{\varepsilon }|^{2^{\ast
}}dx-t^{2^{\ast }-1}\sigma \int_{\mathbb{R}^{N}}\bar{v}U_{\varepsilon
}^{2^{\ast }-1}dx.  \label{e29}
\end{equation}%
Considering $\xi =\Vert \psi _{\varepsilon ,t}\Vert _{2}$, we have
\begin{equation}
\xi ^{2}=1+2t\sigma \int_{\mathbb{R}^{N}}\bar{v}U_{\varepsilon }dx,
\label{e30}
\end{equation}%
which implies that $\xi \geq 1$. A direct calculation shows that
\begin{eqnarray}
&&\xi^{-2}\int_{\mathbb{R}^{N}}\left( V(x/\xi )-\xi ^{2}V(x)\right) |\psi
_{\varepsilon ,t}|^{2}dx  \notag \\
&=&\int_{\mathbb{R}^{N}}V(x)\left[ \xi^{N-2}\psi _{\varepsilon ,t}^{2}(\xi x)-\psi
_{\varepsilon ,t}^{2}(x)\right] dx  \notag \\
&=&\int_{\mathbb{R}^{N}}V(x)\int_{1}^{\xi }\left( \frac{N-2}{2}z^{N-3}\psi
_{\varepsilon ,t}^{2}(zx)+z^{N-2}\psi _{\varepsilon ,t}(zx)\nabla \psi
_{\varepsilon ,t}(zx)\cdot x\right) dzdx  \notag \\
&=&\int_{1}^{\xi }\int_{\mathbb{R}^{N}}z^{-3}V\left( x/z\right) \left( \frac{%
N-2}{2}\psi _{\varepsilon ,t}^{2}(x)+\psi _{\varepsilon ,t}(x)\nabla \psi
_{\varepsilon ,t}(x)\cdot x\right) dxdz.  \label{e31}
\end{eqnarray}%
Define
\begin{equation*}
C_{\mu }=\int_{\mathbb{R}^{N}}V(x)\left( \frac{N-2}{2}\psi _{\varepsilon
,t}^{2}(x)+\psi _{\varepsilon ,t}(x)\nabla \psi _{\varepsilon ,t}(x)\cdot
x\right) dx.
\end{equation*}%
Clearly, $|C_{\mu }|$ is bounded. Then by (\ref{e31}), one has
\begin{eqnarray}
&&\xi^{-2}\int_{\mathbb{R}^{N}}\left( V(x/\xi )-\xi ^{2}V(x)\right) |\psi
_{\varepsilon ,t}|^{2}dx  \notag \\
&=&\int_{1}^{\xi }\left\{ \int_{\mathbb{R}^{N}}\left( z^{-3}V\left(
x/z\right) -V(x)\right) \left( \frac{N-2}{2}\psi _{\varepsilon
,t}^{2}(x)+\psi _{\varepsilon ,t}(x)\nabla \psi _{\varepsilon ,t}(x)\cdot
x\right) dx+C_{\mu }\right\} dz.  \label{e32}
\end{eqnarray}%
According to the definition of $\xi $ and the fact that $1\leq z\leq \xi $,
we conclude that for any sufficiently small $\varepsilon >0$ and for any $%
\ell >0$, there exists $r>0$, independent of $z$, such that
\begin{equation}
\left\vert \int_{\mathbb{R}^{N}\setminus B_{r}(0)}\left( z^{-3}V\left(
x/z\right) -V(x)\right) \left( \frac{N-2}{2}z^{N-2}\psi _{\varepsilon
,t}^{2}(x)+\psi _{\varepsilon ,t}(x)\nabla \psi _{\varepsilon ,t}(x)\cdot
x\right) dx\right\vert <\ell .  \label{e33}
\end{equation}%
Since $V(x)=V_{1}+V_{2}\in L^{N/2}(\mathbb{R}^{N})+L^{\infty }(\mathbb{R}%
^{N})$, it follows from \cite[Lemma 2.7]{VY25} that $\overline{V}=V\cdot \chi
_{B_{r}}\in L^{N/2}(\mathbb{R}^{N})$ and $\overline{W}=W\cdot \chi
_{B_{r}}\in L^{N}(\mathbb{R}^{N})$. Then by Lemma \ref{T3.2} one has
\begin{equation*}
\Vert z^{-3}\overline{V}\left( x/z\right) -\overline{V}(x)\Vert
_{N/2}\rightarrow 0\ \text{and}\ \Vert z^{-3}\overline{W}\left( x/z\right) -%
\overline{W}\Vert _{N/2}\rightarrow 0,
\end{equation*}%
as $z\rightarrow 1$, which implies that
\begin{eqnarray}
&&\left\vert \int_{B_{r}}\left( z^{-3}V\left( x/z\right) -V(x)\right) \left(
\frac{N-2}{2}\psi _{\varepsilon ,t}^{2}(x)+\psi _{\varepsilon ,t}(x)\nabla
\psi _{\varepsilon ,t}(x)\cdot x\right) dx\right\vert   \notag \\
&\leq &\frac{N-2}{2}\Vert z^{-3}\overline{V}\left( x/z\right) -\overline{V}%
(x)\Vert _{L^{N/2}\left( B_{r}\right) }\Vert \psi _{\varepsilon ,t}\Vert
_{2^{\ast }}^{2}  \notag \\
&&+\Vert z^{-2}\overline{W}\left( x/z\right) -\overline{W}\Vert
_{L^{N}\left( B_{r}\right) }\Vert \psi _{\varepsilon ,t}\Vert _{2^{\ast
}}\Vert \nabla \psi _{\varepsilon ,t}\Vert _{2}  \notag \\
&\leq &2\ell .  \label{e34}
\end{eqnarray}%
Thus it follows from (\ref{e30}), (\ref{e32})--(\ref{e34}) that
\begin{equation}
\xi^{-2}\int_{\mathbb{R}^{N}}\left( V(x/\xi )-\xi ^{2}V(x)\right) |\psi
_{\varepsilon ,t}|^{2}dx\leq \int_{1}^{\xi}\left( 3\ell +|C_{\mu }|\right)
dz\leq Ct\int_{\mathbb{R}^{N}}vU_{\varepsilon }dx.  \label{e35}
\end{equation}%
Using the fact that
\begin{equation*}
(a+b)^{p}=a^{p}+pa^{p-1}b+O(b^{p})\text{ for }1<p\leq 2\text{ and }a,b>0,
\end{equation*}%
we calculate that
\begin{equation*}
\left( 1-t^{2}\Vert U_{\varepsilon }\Vert _{2}^{2}\right) ^{2^{\ast
}/2}=
1-\frac{2^{\ast }t^{2}}{2}\Vert U_{\varepsilon }\Vert _{2}^{2}+o(\Vert
U_{\varepsilon }\Vert _{2}^{2})\text{ } \text{for}\ N\geq 6.
\end{equation*}%
Using this, together with (\ref{e16}), leads to
\begin{eqnarray}
J_{\mu }(\sigma \bar{v}) &=&J_{\mu }(\bar{v})-\frac{t^{2}}{2}\Vert
U_{\varepsilon }\Vert _{2}^{2}\left( \Vert \nabla \bar{v}\Vert _{2}^{2}-\mu
\Vert \bar{v}\Vert _{2^{\ast }}^{2^{\ast }}+\int_{\mathbb{R}^{N}}V(x)\bar{v}%
^{2}dx\right)   \notag \\
&&+\frac{\mu }{2^{\ast }}\left( 1-\frac{2^{\ast }}{2}t^{2}\Vert
U_{\varepsilon }\Vert _{2}^{2}-\left( 1-t^{2}\Vert U_{\varepsilon }\Vert
_{2}^{2}\right) ^{\frac{2^{\ast }}{2}}\right) \Vert \bar{v}\Vert _{2^{\ast
}}^{2^{\ast }}  \notag \\
&\leq &\alpha _{\mu }-\frac{t^{2}}{2}\Vert
U_{\varepsilon }\Vert _{2}^{2}\left[\left(1-\Vert V_{1}^{-}\Vert_{\frac{N}{2}}\mathcal{S}^{-1}\right)\Vert\nabla \bar{v}\Vert _{2}^{2}%
-\mu \mathcal{S}^{\frac{-N}{N-2}}\Vert\nabla \bar{v}\Vert _{2}^{2^{\ast}}-\Vert V_{2}^{-}\Vert_{\infty}\right]+o(\Vert U_{\varepsilon }\Vert _{2}^{2}) \notag\\
&\leq &\alpha _{\mu }+\frac{t^{2}}{2}\Vert V_{2}^{-}\Vert _{\infty }\Vert
U_{\varepsilon }\Vert _{2}^{2}+o(\Vert U_{\varepsilon }\Vert _{2}^{2}).
\label{e36}
\end{eqnarray}%
Since $\bar{v}$ is the ground state solution of problem (\ref{e11}) with $%
\lambda =\lambda _{\mu }$ such that $J_{\mu }(\bar{v})=\alpha _{\mu }$, and
$\lambda _{\mu }<0$ by Lemma \ref{T2.6}, we have
\begin{equation}
\int_{\mathbb{R}^{N}}\nabla \bar{v}\nabla U_{\varepsilon }dx-\mu \int_{%
\mathbb{R}^{N}}\bar{v}^{2^{\ast }-1}U_{\varepsilon }dx+\int_{\mathbb{R}%
^{N}}V(x)\bar{v}U_{\varepsilon }dx=\lambda _{\mu }\int_{\mathbb{R}^{N}}\bar{v%
}U_{\varepsilon }dx<0.  \label{e37}
\end{equation}%
By (\ref{e36})--(\ref{e37}), we have
\begin{eqnarray}
&&J_{\mu }(\sigma \bar{v})+t\sigma \int_{\mathbb{R}^{N}}V(x)\bar{v}%
U_{\varepsilon }dx+t\sigma \int_{\mathbb{R}^{N}}\nabla \bar{v}\nabla
U_{\varepsilon }dx  \notag \\
&\leq &\alpha _{\mu }+\frac{t^{2}}{2}\Vert V_{2}^{-}\Vert _{\infty }\Vert
U_{\varepsilon }\Vert _{2}^{2}+\mu \int_{\mathbb{R}^{N}}\bar{v}^{2^{\ast
}-1}U_{\varepsilon }dx+o(\Vert U_{\varepsilon }\Vert _{2}^{2}).
\label{e38}
\end{eqnarray}%
It follows from (\ref{e22})--(\ref{e24}), (\ref{e27})--(\ref{e30}), (\ref{e35})
and (\ref{e38}) that
\begin{eqnarray}
\theta _{\mu } &\leq &\sup_{t\in \lbrack 0,\widetilde{t}]}J_{\mu }(\Psi
_{\varepsilon ,t})  \notag \\
&\leq &\sup_{t\in \lbrack 0,\widetilde{t}]}\left( \frac{t^{2}}{2}\mathcal{S}%
^{\frac{N}{2}}-\frac{\mu t^{2^{\ast }}}{2^{\ast }}\mathcal{S}^{\frac{N}{2}}+%
\frac{t^{2}}{2}\Vert V_{2}^{-}\Vert _{\infty }\Vert U_{\varepsilon }\Vert
_{2}^{2}+\frac{t^{2}}{2}\int_{\mathbb{R}^{N}}V(x)U_{\varepsilon
}^{2}dx\right)\\
&&+(c_5R_{\ast}^{2}-c_6)\varepsilon^{\frac{N-2}{2}}  +\alpha _{\mu }+o(\Vert
U_{\varepsilon }\Vert _{2}^{2}).  \label{e39}
\end{eqnarray}%
Let%
\begin{equation*}
G(t)=\frac{t^{2}}{2}\mathcal{S}^{\frac{N}{2}}-\frac{\mu t^{2^{\ast }}}{%
2^{\ast }}\mathcal{S}^{\frac{N}{2}}+\frac{t^{2}}{2}\Vert V_{2}^{-}\Vert
_{\infty }\Vert U_{\varepsilon }\Vert _{2}^{2}+\frac{t^{2}}{2}\int_{\mathbb{R%
}^{N}}V(x)U_{\varepsilon }^{2}dx\text{ for }0<t<\overline{t}.
\end{equation*}%
By (\ref{e8}), we deduce that for $t>0$,
\begin{equation*}
G(t)<\frac{t^{2}}{2}\mathcal{S}^{\frac{N}{2}}-\frac{\mu t^{2^{\ast }}}{%
2^{\ast }}\mathcal{S}^{\frac{N}{2}}-Ct^{2}\Vert U_{\varepsilon }\Vert
_{2}^{2}\ \text{for}\ N\geq 6,
\end{equation*}%
which implies that
\begin{equation}
\sup\limits_{t\in \lbrack 0,\bar{t}]}G(t)<\frac{1}{N}\mathcal{S}^{%
\frac{N}{2}}\mu ^{1-\frac{N}{2}}-C\Vert U_{\varepsilon }\Vert _{2}^{2}\text{
for}\ N\geq 6.  \label{e40}
\end{equation}%
Therefore, by (\ref{e22}) and (\ref{e39})--(\ref{e40}), we conclude
that $\theta _{\mu }<\alpha _{\mu }+\frac{1}{N}\mathcal{S}^{\frac{N}{2}}\mu
^{1-\frac{N}{2}}$, where we have controlled $R_{\ast}>0$ small enough such that $c_5R_{\ast}^{2}-c_6<0$. The proof is complete.
\end{proof}

To establish the existence result, we need the min-max principle developed
by N. Ghoussoub \cite[Theorem 4.1]{G93} to construct a mountain pass type
critical point for $J_{\mu }|_{\mathcal{M}}$.

\begin{lemma}
\label{T3.4} (\cite[Theorem 4.1]{G93}) Let $X$ be a Hilbert manifold and let
$E_{\mu }\in C^{1}(M,\mathbb{R})$ be a given functional. Let $K\subset X$ be
compact and consider a subset
\begin{equation*}
\Gamma \subset \{\gamma \subset X:\gamma \ \text{is compact,}\ K\subset
\gamma \},
\end{equation*}%
which is invariant with respect to deformations leaving $K$ fixed. Assume
that
\begin{equation*}
\max\limits_{v\in K}J_{\mu }(v)<\theta _{\mu }:=\inf\limits_{\gamma \in
\Gamma }\max\limits_{v\in \gamma }J_{\mu }(v).
\end{equation*}%
Let $(\gamma _{n})_{n}\subset \Gamma $ be a sequence such that $%
\max\limits_{v\in \gamma _{n}}J_{\mu }(v)\rightarrow \theta _{\mu }\ $as$\
n\rightarrow +\infty .$ Then there exists a sequence $(v_{n})_{n}\subset X$
such that, as $n\rightarrow +\infty $,
\begin{equation*}
J_{\mu }(v_{n})\rightarrow \theta _{\mu },\ \Vert \nabla _{X}J_{\mu
}(v_{n})\Vert \rightarrow 0\ \text{and }dist(v_{n},\gamma _{n})\rightarrow 0.
\end{equation*}
\end{lemma}

\begin{lemma}
\label{T3.5} There exists a bounded Palais-Smale sequence $\left(
v_{n},\lambda _{n}\right) $ for $J_{\mu }$ on $\mathcal{M}$ at the level $%
\theta _{\mu }$, satisfying
\begin{equation}
J_{\mu }\left( v_{n}\right) \rightarrow \theta _{\mu },\ J_{\mu }^{\prime
}\left( v_{n}\right) |_{\mathcal{M}}\rightarrow 0,\ \text{as}\ n\rightarrow
+\infty ,\ \lim\limits_{n\rightarrow +\infty }\Vert v_{n}^{-}\Vert _{2}=0,
\label{e41}
\end{equation}%
and
\begin{equation}
\Vert v_{n}\Vert _{2}^{2}-\mu \Vert v_{n}\Vert _{2^{\ast }}^{2^{\ast }}+%
\frac{1}{2}\int_{\mathbb{R}^{N}}V(x)\left( Nv_{n}^{2}+2v_{n}\nabla
v_{n}\cdot x\right) dx\rightarrow 0,\ \text{as}\ n\rightarrow \infty .
\label{e42}
\end{equation}%
Moreover,  there exist $\bar{\lambda}_{\ast}\in \mathbb{R}, \bar{v}_{\ast }\in H^{1}(\mathbb{R}^{N})\setminus
\{0\}$ such that up to a subsequence, $\lambda_{n}\rightarrow \bar{\lambda}_{\ast }$ in $%
\mathbb{R}$ and $v_{n}\rightarrow \bar{v}_{\ast }$ in $%
H^{1}(\mathbb{R}^{N})$.
\end{lemma}

\begin{proof}
According to Lemma \ref{T3.4}, we obtain a Palais-Smale sequence $\left(
v_{n},\lambda _{n}\right) $ which satisfies (\ref{e41}) and (\ref{e42}).
Similar to \cite[Lemma 5.6]{VY25}, we can prove that this sequence is
bounded. Finally, it follows from Lemmas \ref{T2.3} and \ref{T3.3} that
there exist $ \bar{\lambda}_{\ast}\in \mathbb{R}, \bar{v}_{\ast }\in H^{1}(\mathbb{R}^{N})\setminus
\{0\}$ such that up to a subsequence, $\lambda_{n}\rightarrow \bar{\lambda}_{\ast }$ in $%
\mathbb{R}$ and $v_{n}\rightarrow \bar{v}_{\ast }$ in $%
H^{1}(\mathbb{R}^{N})$. The proof is complete.
\end{proof}

\textbf{We are ready to prove Theorem \ref{T1.4}.} It directly follows from
Lemma \ref{T3.5}.

\section*{Acknowledgments}

J. Sun was supported by the National Natural Science Foundation of China
(Grant No. 12371174) and Shandong Provincial Natural Science Foundation
(Grant No. ZR2020JQ01).


\begin{thebibliography}{99}

\bibitem{BMRV21} T. Bartsch, R. Molle, M. Rizzi, G. Verzini, Normalzied
solutions of mass supercritical Schr\"{o}dinger equations with potential,
Comm. Partial Differential Equations 46 (2021) 1729--1756.

\bibitem{BM21} B. Bieganowski, J. Mederski, Normalized ground states of the
nonlinear Schr\"{o}dinger equation with at least mass critical growth, J.
Funct. Anal. 280 (2021) 108989.

\bibitem{BN83} H. Br\'{e}zis, E.H. Lieb, A relation between pointwise
convergence of functions and convergence of functionals, Proc. Amer. Math.
Soc. 88 (1983) 486--490.

\bibitem{DZ22} Y. Ding and X. Zhong, Normalized solution to the Schr\"{o}%
dinger equation with potential and general nonlinear term: Mass
super-critical case, J. Differential Equations 334 (2022) 194--215.

\bibitem{G93} N. Ghoussoub, Duality and perturbation methods in critical
point theory, Cambridge University Press, Cambridge, 1993.

\bibitem{IM20} N. Ikoma, Y. Miyamoto, Stable standing waves of nonlinear Schr%
\"{o}dinger equations with potentials and general nonlinearities, Calc. Var.
PDE 59 (2020) 48.

\bibitem{J97} L. Jeanjean, Existence of solutions with prescribed norm for
semilinear elliptic equations, Nonlinear Anal. 28 (1997) 1633--1659.

\bibitem{JL22} L. Jeanjean, T.T. Le, Multiple normalized solutions for a
Sobolev critical Schr\"{o}dinger equation, Math. Ann. 384 (2022) 101--134.

\bibitem{L84} P.L. Lions, The concentration-compactness principle in the
calculus of variations. The locally compact case. Part I and II. In: Ann.
Inst. H. Poincar\'{e} Anal. Non Lin\'{e}aire 1 (1984) 223--283.

\bibitem{MRV22} R. Molle, G. Riey, G. Verzini, Normalized solutions to
mass supercritical Schr\"{o}dinger equations with negative potential, J.
Differential Equations 333 (2022) 302--331.

\bibitem{PPVV21} B. Pellacci, A. Pistoia, G. Vaira, G. Verzini, Normalized concentrating solutions to nonlinear elliptic problems, J.
Differential Equations 275 (2021) 882--919.

\bibitem{PV17} D. Pierotti, G. Verzini, Normalized bound states for the nonlinear Schr\"{o}dinger equation in bounded domains, Calc. Var. PDE 56
(2017) 133.

\bibitem{PVY24} D. Pierotti, G. Verzini, J. Yu, Normalized solutions for Sobolev critical Schr\"{o}dinger equations on bounded domains, SIAM J. Math.
Anal. 57 (2025) 262--285.

\bibitem{S120} N. Soave, Normalized ground states for the NLS equation with combined nonlinearities, J. Differential Equations 269 (2020) 6941--6987.

\bibitem{S20} N. Soave, Normalized ground states for the NLS equation with combined nonlinearities: the Sobolev critical case, J. Funct. Anal. 279
(2020) 108610.

\bibitem{SZ24} L. Song, W. Zou, Two positive normalized solutions on star-shaped bouned domains to the Br\'{e}zis-Nirenberg problem, I: Existence, arXiv: 2404.11204v1, (2024).

\bibitem{SZRW25} J. Sun, J. Zhang, V.D. Radulescu, T.F. Wu, Choquard equations with saturable reaction, Calc. Var. PDE (2025) 64:61.

\bibitem{SYZ24} J. Sun, S. Yao, J. Zhang, Planar Schr\"{o}dinger--Poisson system with exponential critical growth: Local well-posedness and standing waves with prescribed mass, Stud. Appl. Math. (2024) 153: e12760.

\bibitem{VY25} G. Verzini, J. Yu, Normalized solutions for the nonlinear Schr\"{o}dinger equation with potential: the purely Sobolev critical case, arXiv:2505.05357v1, (2025).

\bibitem{WS23} C. Wang, J. Sun, Normalized solutions for the p-Laplacian equation with a trapping potential, Adv. Nonlinear Anal. (2023) 12: 20220291.

\bibitem{WW22} J. Wei, Y. Wu, Normalized solutions for Schr\"{o}dinger equations with critical Sobolev exponent and mixed nonlinearities, J. Funct. Anal. 283 (2022) 109574.

\bibitem{ZYS25} R. Zhang, S. Yao, J. Sun, Normalized solutions for Schr\"{o}dinger-Bopp-Podolsky system with a negative potential, Appl. Math. Lett. 161 (2025) 109368.

\bibitem{ZZ23} X. Zhong, W. Zou, A new deduction of the strict binding inequality and it application: Ground state normalized solution to Schr\"{o}dinger equations with potential, Differ. Integral Equ. 36 (2023) 133--160.
\end{thebibliography}
\end{document}